\documentclass[a4paper]{article}
\usepackage[margin=.8in]{geometry}
\usepackage{indentfirst}
\title{Bounds on the Semidefinite Programming Complexity of the Euclidean Ball}

\author{Guanxi Li\footnote{Georgia Institute of Technology, Department of Mathematics, \texttt{gli443@gatech.edu}} and Kevin Shu\footnote{University of Waterloo, Combinatorics and Optimization, \texttt{k6shu@uwaterloo.ca}}}

\usepackage{amssymb,amsmath,amsthm}
\usepackage{enumerate}
\usepackage{hyperref}
\usepackage[capitalize]{cleveref}
\usepackage{enumitem}
\usepackage{color}
\usepackage{cancel}
\usepackage[all, cmtip]{xy}
\usepackage{tikz}
\usetikzlibrary{shapes}
\usepackage{svg}
\usepackage{import}
\usepackage{optidef}

\usepackage[utf8]{inputenc} 
\usepackage[T1]{fontenc}    
\usepackage{hyperref}       
\usepackage{url}            
\usepackage{booktabs}       
\usepackage{amsfonts}       
\usepackage{nicefrac}       
\usepackage{microtype}
\usepackage{multirow}
\usepackage{xfrac}
\usepackage{titlesec}
\usepackage{caption}
\usepackage{subcaption}
\usepackage{amsmath} 
\usepackage{float}
\usepackage{algorithm}
\usepackage{algpseudocode}
\usepackage{multicol}

\titleformat{\subsubsection}
{\normalfont\normalsize\bfseries}{\thesubsubsection}{1em}{}

\usepackage{graphicx}

\usepackage{appendix}
\numberwithin{equation}{section}

\newcommand{\tr}{\mathrm{Tr}}

\newcommand{\inclu}[0] {\ar@{^{(}->}}

\newcommand{\R}{\mathbb{R}}

\newcommand{\N}{\mathbb{N}}

\newtheorem{theorem}{Theorem}[section]

\crefname{claim}{claim}{claims}
\Crefname{claim}{Claim}{Claims}
\crefname{lem}{lemma}{lemmas}
\Crefname{lem}{Lemma}{Lemmas}
\crefname{algorithm}{algorithm}{algorithms}
\Crefname{algorithm}{Algorithm}{Algorithms}

\usepackage{mathtools}

\usepackage[T1]{fontenc}
\usepackage{lmodern}

\DeclareMathOperator{\Gr}{Gr}

\DeclareMathOperator{\Ext}{Ext}
\DeclareMathOperator{\sxc}{sxc}
\DeclareMathOperator{\lmi}{lmi}
\DeclareMathOperator{\Graph}{Graph}

\newcommand{\Sym}{\mathbb{S}}

\begin{document}
\maketitle

\begin{abstract}
    We present tight lower bounds for the size of any linear matrix inequality representation of the Euclidean ball, and on the semidefinite extension complexity of the ball. Specifically, we show that any linear matrix inequality representing the Euclidean ball must involve matrices of size at least $n$, and any spectrahedral extension of the Euclidean ball must involve matrices of size at least $\lceil 2\sqrt{n-1}\rceil$. These match the upper bounds given by existing explicit constructions exactly. Our proofs rely on elementary facts about the dimension of the set of faces of spectrahedra, and indeed extend to any convex set whose extreme points form a semialgebraic set.
\end{abstract}

\section{Introduction}
Understanding the sizes of semidefinite programming descriptions of convex sets is of fundamental importance in optimization \cite{FawziGouveiaParriloSaundersonThomas2022}. Here, a semidefinite programming description of a convex set is either a representation of that set in terms of a linear matrix inequality, or an extended formulation of that set that is given by a linear matrix inequality.

However, lower bounds on the sizes of such descriptions are generally difficult to obtain. Indeed, even in the case of the Euclidean ball, it was previously unknown what the smallest size of a description of this set in either sense is. The best previously known lower bounds only showed that for infinitely many values of $n$, the smallest size of a linear matrix inequality representing the $n$ dimensional Euclidean ball is exactly $n$, and the semidefinite programming extension complexity of the Euclidean ball was not known exactly for any $n > 2$ \cite{Kummer2016}. We show that relatively elementary geometric considerations about the Euclidean ball are sufficient to give optimal lower bounds on the semidefinite extension complexity of the Euclidean ball.

Specifically, we note that the set of extreme points of the $n$ dimensional Euclidean ball $B_n \subseteq \R^n$ is the $n-1$ dimensional sphere.
On the other hand, it is well known that the faces of the positive semidefinite cone of a given rank $k$ are parameterized by an algebraic variety known as the Grassmannian.
By comparing the dimensions of the Grassmannian varieties to the dimension of the sphere, we are able to obtain tight bounds on the size of any linear matrix inequality (LMI) description of the ball, and also tight bounds on the semidefinite extension complexity of the ball. 

In fact, our results extend to any convex set with a semialgebraic set of extreme points after taking into account the dimension of its extreme points. This more general result may be of independent interest.
\subsection{Preliminaries and Statement of Main Results}

A \emph{linear matrix inequality} (LMI) is an inequality of the form 
\[
    A_0 + A_1 x_1 + \dots + A_n x_n \succeq 0,
\]
where $A_0, \dots, A_n \in \Sym^{r}$. A set of the form $\{x \in \R^n : A_0 + A_1 x_1 + \dots + A_n x_n \succeq 0\}$ is called a \emph{spectrahedron}, and we say that it has a LMI description of size $r$. We will denote by $\lmi(C)$ the smallest $r$ so that $C$ has an LMI description of size $r$. The \emph{semidefinite extension complexity} of a set $C \subseteq \R^n$ is then the smallest $r$ such that there exists a spectrahedron $P \subseteq \R^m$ with $\lmi(P) = r$ together with a linear map $\pi : \R^m \rightarrow \R^n$ so that $\pi(P) = C$.
We will let $\Ext(C)$ denote the set of extreme points of a convex set $C$, and we will let $\sxc(C)$ denote its semidefinite extension complexity. We will also let $\Sym^r_+$ denote the positive semidefinite cone in $\Sym^r$.

\begin{theorem}\label{thm:lmi_bound}
    Let $C \subseteq \R^n$ be a convex body containing $0$ in its interior and let $C^*$ denote its polar. If $\Ext(C^*)$ is a semialgebraic set of dimension $d$, then $\lmi(C) \ge d+1$.
\end{theorem}

\begin{theorem}\label{thm:sxc_bound}
    Let $C \subseteq \R^n$ be a convex body. If $\Ext(C)$ is a semialgebraic set of dimension $d$, then the semidefinite extension complexity of $C$ satisfies
    \[
        \sxc(C) \ge 2 \sqrt{d}.
    \]
\end{theorem}

\subsection{Semidefinite Programming Descriptions of the Ball}

Given our main theorems, we can now give optimal bounds for both the smallest size of an LMI description of the Euclidean ball, and its semidefinite extension complexity.

\begin{theorem}
    The smallest size of an LMI description of the Euclidean ball $B_n$ is $n$ for all $n \ge 2$.
\end{theorem}
\begin{proof}
    The lower bound is implied by \cref{thm:lmi_bound} by noting that the polar of the Euclidean ball is itself, and that it has an $n-1$ dimensional set of extreme points.

    The upper bound is supplied by the standard LMI description of size $n$ \cite{Kummer2016}:
    \[
        B_n = \left\{x \in \R^n :
            \begin{pmatrix} 1+x_1 & x_2   & x_3   & \dots & x_n \\
                            x_2   & 1-x_1 & 0     & \dots & 0 \\
                            x_3   & 0     & 1-x_1 & \dots & 0\\
                                  &       &       & \dots & \\
                            x_n  & 0      & 0     & \dots & 1-x_1
                    \end{pmatrix} \succeq 0\right\}.
    \]
\end{proof}

\begin{theorem}
    The semidefinite extension complexity of the Euclidean ball $B_n$ is $\lceil 2 \sqrt{n-1}\rceil$ for all $n \ge 2$.
\end{theorem}
\begin{proof}
    The lower bound is supplied directly by \cref{thm:sxc_bound}.

    On the other hand, to upper bound the semidefinite extension complexity of $B_n$, we will provide the following construction, which is a slight modification of one that appears in \cite[Example 1.5]{thomas2018spectrahedral}.
    For any $a, b \in \N$ so that $ab \ge n-1$, we will see that
    \begin{equation}
        B_n = \left\{x \in \R^n : \exists Y \in \R^{b \times b}, \tr(Y) = 1-x_1, \begin{pmatrix} (1+x_1)I_a & M(x_{-1}) \\ M(x_{-1})^{\intercal} & Y \end{pmatrix} \succeq 0\right\}, \label{eq:sxc_ball}
    \end{equation}
    where $x_{-1}$ denotes the vector obtained by deleting the first entry of $x$, and $M(x)$ is any norm preserving linear map $M : \R^{n-1} \rightarrow \R^{a \times b}$.
    Given this description of the ball, and by letting $k = \lceil 2 \sqrt{n-1}\rceil$, $a = \lfloor \frac{k}{2}\rfloor$, and $b = \lceil \frac{k}{2}\rceil$, we obtain the bound that $\sxc(B_n) \le \lceil 2\sqrt{n-1} \rceil$.

    We will now argue that this description is correct. Let $C$ be the projected spectrahedron on the right side of \cref{eq:sxc_ball}. For some $x$ with $x_1 \neq -1$, we can see that $x \in B_n$ if and only if $x \in C$ as follows.
    When $x_1 \neq -1$, we can apply the Schur complement to the LMI defining $C$ to obtain the equivalent inequality
    \[
        \begin{pmatrix} (1+x_1)I_a & M(x_{-1}) \\ M(x_{-1})^{\intercal} & Y \end{pmatrix} \succeq 0 \Leftrightarrow
        \frac{1}{1+x_1}M(x_{-1})^{\intercal}M(x_{-1}) \preceq Y. 
    \]
    If this inequality holds, then
    \[
        \tr\left(\frac{1}{1+x_1}M(x_{-1})^{\intercal}M(x_{-1})\right) = \frac{1}{1+x_1}\|x_{-1}\|^2 \le \tr(Y) = 1-x_1,
    \]
    implying that $\|x\|^2 \le 1$.
    On the other hand, if $\|x\|^2 \le 1$, then letting $Y = \frac{1}{1+x_1}M(x_{-1})^{\intercal}M(x_{-1}) + \frac{1-\|x\|^2}{b(1+x_1)}I_b$ implies that $\frac{1}{1+x_1}M(x_{-1})^{\intercal}M(x_{-1}) \preceq Y$ and also $\tr(Y) = 1-x_1$.

    In the case when $x_1 = -1$, then we see that the LMI reduces to
    \[
        \begin{pmatrix} 0 & M(x_{-1}) \\ M(x_{-1})^{\intercal} & Y \end{pmatrix} \succeq 0
    \]
    This is the case if and only if $x_{-1} = 0$ and $Y$ is any PSD matrix with trace 2.
\end{proof}

Note that the lower bound on semidefinite extension complexity given by comparing the dimension of $B_n$ to the dimension of $\Sym^r_+$ only yields $\sxc(B_n) \ge \frac{\sqrt{8n+1}-1}{2}$.

\section{Real Algebraic Geometry Preliminaries}
Here, we will recall some basic facts about semialgebraic sets. In general, a semialgebraic set is a subset of $\R^n$ that can be defined by a collection of polynomial equations and inequalities.
If $A$ and $B$ are semialgebraic sets and $f : A \rightarrow B$ is a function, we say that $f$ is semialgebraic if its graph $\Graph f = \{(a,f(a)) : a \in A\}$ is semialgebraic.

Any semialgebraic set has a dimension, which is the largest dimension of a smooth manifold contained in that set. For references on the dimension theory of semialgebraic sets, see \cite[Chapter 2.8]{BochnakCosteRoy1998}. For our purposes, we will only need the fact that if $A \subseteq B$ are semialgebraic sets, then the dimension of $A$ is at most that of $B$, and the fact that the image of a semialgebraic set under a semialgebraic map is at most the dimension of its domain.

We will also make use of the Tarski-Seidenberg theorem, which states that if $A \subseteq \R^{d_1 + d_2}$ is a semialgebraic set, then $\{x \in \R^{d_1}: \exists y \in \R^{d_2},\;(x,y) \in A\}$ is also semialgebraic. This is also known as quantifier elimination for real closed fields.

Finally, we will need to know about the total Grassmannian, which is defined as the set of all subspaces of $\R^r$. There are different representations of the Grassmannian, as a real algebraic variety (which is a semialgebraic set defined by only polynomial equations), but the total real Grassmannian $\Gr^r$ can be defined as follows:
\[
    \Gr^r = \{S \in \Sym^r : S^2 = S\}.
\]
Here, we identify a subspace of $\R^r$ by the unique matrix $S \in \Gr^r$ whose image is that subspace. Typically, the Grassmannian is partitioned into pieces graded by rank, i.e.
\[
    \Gr^{r,k} = \{S \in \Sym^r : S^2 = S, \tr(S) = k\}.
\]
It is known that $\Gr^{r,k}$ is an algebraic variety of dimension $k(r-k)$, and so the dimension of $\Gr^r = \bigcup_{k=0}^r \Gr^{r,k}$ is $\max_{0 \le k \le r}k(r-k) \le \frac{r^2}{4}$.

\section{Proof of \Cref{thm:lmi_bound}}
We will assume that we have the representation
\[
    C = \{x \in \R^n : A_0 + A_1 x_1 + \dots + A_n x_n \succeq 0\},
\]
where each $A_i \in \Sym^{r}$.
The fact that $0 \in C$ implies that $A_0 \succeq 0$. We can then without loss of generality assume that $A_0$ is of full rank (by possibly making the LMI description smaller by removing the common kernel of these matrices). By applying a congruence transformation to the $A_i$ assume that $A_0 = I$.

It then follows from elementary facts about the polar and semidefinite programming duality that we have the following description of $C^*$:
\[
    C^* = \{(-\langle A_1, X\rangle, \dots, -\langle A_n, X\rangle) : X \in \Sym_+^{r}, \tr(X) \le 1\}.
\]
The set of extreme points of $E_r = \{X \in \Sym^{r}_+ : \tr(X) \le 1\}$ is exactly $\{xx^{\intercal} : x \in \R^r, \|x\| = 1\} \cup \{0\}$, which is of dimension $r-1$. Thus, the linear map $A(X) = (-\langle A_1, X\rangle, \dots, -\langle A_n, X\rangle)$ restricts to a map from $\Ext(E_r)$ to $\R^n$ so that $\Ext(C^*) \subseteq A(\Ext(E_r))$, and since linear maps cannot increase the dimension of a semialgebraic set, it follows that $r-1 \ge d$, which implies our desired result.

\section{Extension Complexity}

\subsection{Proof of \cref{thm:sxc_bound}}
This proof can be viewed as an extension of the proof of \Cref{thm:lmi_bound}.
For this, we note that $\sxc(C) \le r$ if and only if there exists a linear map $\pi : \Sym^r \rightarrow \R^n$ and an affine subspace $\mathcal{L} \subseteq \Sym^r$ so that 
\[
    C = \pi(\mathcal{L} \cap \Sym^r_+).
\]
We will let $P$ denote $\mathcal{L} \cap \Sym^r_+$.

For any extreme point $x \in \Ext(C)$, we see that $\pi^{-1}(x) \cap P$ is some face of $P$.
There is a function $f$ from the set of faces $F$ of $P$ so that $\pi(F)$ is a singleton to $\R^n$, which sends any such $F$ to the unique element of $\pi(F)$. We can see that the image of this function $f$ contains the extreme points of $C$. We will use this map to compare the dimension of the set of faces of $P$ to that of $\Ext(C)$, but first we need to parameterize the set of faces of $P$ in such a way that we can discuss the dimension of this set.

Every face of $\Sym^r_+$ is of the form $F_S = \{X \in \Sym^r_+ :  S \subseteq \ker(X)\}$ for some linear subspace $S\subseteq \R^r$. It follows that every face of $P$ is of the form $F_S \cap P$ for some linear subspace $S \subseteq \R^r$.
Identifying $S$ with an element of $\Gr^r$, we can equivalently write $F_S \cap P = \{X \in P : SX = 0\}$.

We will let $Y$ denote the set of elements $S$ of $\Gr^r$ so that $\pi(F_S\cap P)$ is a singleton, and we will let $f : Y \rightarrow \R^n$ be the map sending $S \in Y$ to the unique element of $\pi(F_S\cap P)$. It is therefore clear that $\Ext(C) \subseteq f(Y)$, since every extreme point $x \in \Ext(C)$ is the image of the $S$ representing $\pi^{-1}(x)$ under the map $f$.

At this point, we will note that $Y \subseteq \Gr^r$, and so if $Y$ is semialgebraic, then the dimension of $Y$ is at most that of $\Gr^r$. If the map $f$ is also semialgebraic, then the dimension of $f(Y)$ will be at most the dimension of $Y$. Combining these observations together, we would conclude that 
\[
    \frac{r^2}{4} \ge \dim \Gr^r \ge \dim Y \ge \dim f(Y) \ge \dim \Ext(C) = d.
\] 
From this, we conclude that $r \ge 2 \sqrt{d}$, as desired. 

It therefore remains to show that $Y$ is semialgebraic and that $f$ is semialgebraic, which is the purpose of the remainder of this section. To show this, we will make use of the Tarski-Seidenberg theorem.

Thus, the following representation of $Y$ is sufficient to see that it is a semialgebraic set:
\[
    Y = \{S \in \Gr^r : \exists A \in P, SA = 0\} \setminus \{S \in \Gr^r : \exists A,B \in P, SA = SB = 0, \pi(A) \neq \pi(B) \}.
\]

And the following representation of the graph of $f$ is sufficient to see that $f$ is semialgebraic.
\[
    \Graph f = \{(S, x) : S \in Y,\;\exists A \in P, SA = 0, x = \pi(A)\}.
\]

\subsection{The Extension Complexity of the Euclidean Ball}
It is worth reflecting on the steps that this argument undergoes in the case when $C = B_n$ is the Euclidean ball. Because we have seen that this lower bound is in fact tight for the Euclidean ball, this suggests that the construction giving the upper bound should saturate the bounds in this section.

Recall that we have $M : \R^{n-1} \rightarrow \R^{a\times b}$ is a norm preserving map, where $a = \lfloor \frac{\lceil 2 \sqrt{n-1}\rceil}{2}\rfloor$ and $b= \lceil \frac{\lceil 2 \sqrt{n-1}\rceil}{2}\rceil$.
We can see that when $\|x\| = 1$, the only way that 
\[
    \begin{pmatrix} (1+x_1)I_a & M(x_{-1}) \\ M(x_{-1})^{\intercal} & Y \end{pmatrix} \succeq 0
\]
with $\tr(Y) = 1-x_1$ is if $Y = \frac{M(x_{-1})^{\intercal}M(x_{-1})}{1+x_1}$. If this is the case, then the matrix factors as 
\[
    \begin{pmatrix} (1+x_1)I_a & M(x_{-1}) \\ M(x_{-1})^{\intercal} & \frac{M(x_{-1})^{\intercal}M(x_{-1})}{1+x_1} \end{pmatrix} =
    \begin{pmatrix} \sqrt{(1+x_1)}I_a \\ \frac{1}{\sqrt{1+x_1}}M(x_{-1})^{\intercal}\end{pmatrix}\begin{pmatrix} \sqrt{(1+x_1)}I_a \\ \frac{1}{\sqrt{1+x_1}}M(x_{-1})^{\intercal}\end{pmatrix}^{\intercal}.
\]
The factors of this matrix are in what algebraic geometers call the standard affine chart for the Grassmannian.

The kernel of this matrix is
\[
    S = \text{range }\begin{pmatrix} -\frac{1}{\sqrt{1+x_1}}M(x_{-1}) \\ \sqrt{(1+x_1)}I_b\end{pmatrix},
\]
and so this is the intersection between the feasible region of the extended formulation and $F_S$.

If $n = ab+1$, then the map $M$ is an isomorphism, and a dense set of subspaces of dimension $b$ can be represented in this fashion. In particular, $Y$, the semialgebraic subset of $\Gr^{a+b}$ defined in the proof of \cref{thm:sxc_bound}, is full dimensional inside of $\Gr^{a+b}$.
If $n \neq ab+1$ though, then it may be the case that $Y$ is of positive codimension in $\Gr^{a+b}$, however, our argument shows that the dimension of $Y$ is still larger than the dimension of $\Gr^{c}$ for $c < a+b$.

\section{Conclusions}
We have seen that relatively elementary considerations give optimal lower bounds on the size of LMI descriptions of the Euclidean ball and also its semidefinite extension complexity.

Moreover, the same lower bounds hold for any convex set whose extreme points form a semialgebraic hypersurface.
This includes any $L_p$ ball for $1 < p < \infty$ rational, a large number of rigidly convex sets, and the convex hulls of many algebraic hypersurfaces.
However, in these settings, it is unlikely that our results will yield tight results on their semidefinite extension complexity, as our result only differs from the trivial lower bound obtained by comparing the dimension of the semidefinite cone to that of the desired convex body by a constant factor.

This perspective does open up the possibility that perhaps we can use more sophisticated invariants to compare the facial structure of the positive semidefinite cone to that of a given set for which we wish to construct a semidefinite extension.
For example, it may be worthwhile to study other geometric invariants of the set of extreme points of a given convex set, and see if that leads to obstructions to semidefinite programming lifts.

It would also be interesting to understand whether this argument can be interpreted in terms of the slack operator formulation of \cite{G13}. There are some technical challenges in translating the geometric argument given in the proof of \cref{thm:sxc_bound} into the slack operator language, but it may be the case that doing this translation could suggest useful generalizations and strengthenings of these results.

\section{AI Disclosure}
ChatGPT was used to prove the bounds on the size of LMI description of the Euclidean ball. The generalization to other convex sets, and also to bound extension complexity was then done by hand.
ChatGPT was also used for drafting the paper and literature search.

\bibliographystyle{plain}
\bibliography{euclidean_ball_citations.bib}

@article{FawziGouveiaParriloSaundersonThomas2022,
  author  = {Fawzi, Hamza and Gouveia, Jo{\~a}o and Parrilo, Pablo A. and Saunderson, James and Thomas, Rekha R.},
  title   = {Lifting for Simplicity: Concise Descriptions of Convex Sets},
  journal = {SIAM Review},
  volume  = {64},
  number  = {4},
  pages   = {866--918},
  year    = {2022},
  doi     = {10.1137/20M1329181}
}

@article{Kummer2016,
  author  = {Kummer, Mario},
  title   = {Two Results on the Size of Spectrahedral Descriptions},
  journal = {SIAM Journal on Optimization},
  volume  = {26},
  number  = {1},
  pages   = {589--601},
  year    = {2016},
  doi     = {10.1137/15M1030789}
}

@book{BochnakCosteRoy1998,
  author    = {Bochnak, Jacek and Coste, Michel and Roy, Marie-Fran{\c{c}}oise},
  title     = {Real Algebraic Geometry},
  series    = {Ergebnisse der Mathematik und ihrer Grenzgebiete. 3. Folge},
  volume    = {36},
  publisher = {Springer},
  address   = {Berlin},
  year      = {1998},
  doi       = {10.1007/978-3-662-03718-8}
}

@inproceedings{thomas2018spectrahedral,
  title={Spectrahedral lifts of convex sets},
  author={Thomas, Rekha R},
  booktitle={Proceedings of the International Congress of Mathematicians: Rio de Janeiro 2018},
  pages={3819--3842},
  year={2018},
  organization={World Scientific}
}

@article{G13,
   title={Lifts of Convex Sets and Cone Factorizations},
   volume={38},
   ISSN={1526-5471},
   url={http://dx.doi.org/10.1287/moor.1120.0575},
   DOI={10.1287/moor.1120.0575},
   number={2},
   journal={Mathematics of Operations Research},
   publisher={Institute for Operations Research and the Management Sciences (INFORMS)},
   author={Gouveia, João and Parrilo, Pablo A. and Thomas, Rekha R.},
   year={2013},
   month=May, pages={248–264} }
\end{document}